\documentclass{article}
\usepackage{exptex}
\usepackage{expthmi}
\usepackage{xcolor}
\numberwithin{equation}{section}
\usepackage{enumerate}
\usepackage[numbers]{natbib}
\setcitestyle{numbers}

\newtheorem{thm}{Theorem}
\newtheorem{lem}[thm]{Lemma}

\begin{document}
\title{Discrete Operational  Calculus in  Delayed Stochastic Games}
\author{Jewgeni H. Dshalalow\\
{\color{blue}eugene@fit.edu}
\and
Kenneth Iwezulu\\
{\color{blue}kiwezulu2012@my.fit.edu}
\and
Ryan T. White\\
{\color{blue}rwhite2009@fit.edu}}
\date{}
\maketitle

\vspace{-1cm}
\begin{center}
Department of Mathematical Sciences\\
College of Science\\
Florida Institute of Technology\\
Melbourne, Florida 32901, USA
\end{center}

\begin{abstract}
This article deals with classes of antagonistic games with two players. A game is specified in terms of two ``hostile'' stochastic processes representing mutual attacks upon random times exerting casualties of random magnitudes. The game ends when one of the players is defeated. We target the first passage time $\tau_\rho$ of the defeat and the amount of casualties to either player upon $\tau_\rho$. Here we validate our claim of analytic tractability of the general formulas obtained in \cite{Agarwal:2005uq} under various transforms.

\textbf{Keywords}: Noncooperative stochastic games, marked point processes, Poisson process, fluctuation theory, ruin time, exit time, first passage time, modified Bessel functions.

\textbf{AMS Subject Classification}: 82B41, 60G51, 60G55, 60G57, 91A10, 91A05, 91A60, 60K05.
\end{abstract}

\section{Introduction}
In this paper we model purely antagonistic stochastic games of two players, A and B, who periodically attack each other according to two independent marked random measures
\begin{align}
\mathcal{A}:=\sum_{j\geq 1}w_j\varepsilon_{s_j},\hspace{1cm}\text{ and  }\hspace{1cm} \mathcal{B}:=\sum_{k\geq 1}z_k\varepsilon_{t_k},
\end{align}
where $s_1>0$, $t_1>0$.

The game evolves as a mutual conflict involving two players A and B hitting each other at random until one of the players is ``exhausted.'' In short, the players attack each other in accordance with two independent marked point processes $\mathcal{A}$ and $\mathcal{B}$ of (1.1) on a probability space $\left(\Omega,\mathcal{F},P\right)$, where $\varepsilon_a$ is the Dirac point mass at point $a\in\mathbb{R}$, $\sum_{j\geq 1}\varepsilon_{s_j}$ and $\sum_{k\geq 1}\varepsilon_{t_k}$ are underlying point random measures representing the times of attacks, and the marks $\{w_j\}$ and $\{z_k\}$ (nonnegative random variables) represent respective amounts of damage dealt to players A and B. Players A and B can sustain the attacks until their respective cumulative casualties cross thresholds $M$ and $N$ (positive real numbers). At a time when it takes place (called the first passage time), i.e. when one of the players loses the game, the game should formally stop.

However, the game is observed upon random epochs of time $\mathcal{T}=\{\tau_1,\tau_2,\ldots\}$ and the outcome of the game is not known in real time. The first passage time is then shifted to an epoch $\tau_\rho$, i.e. upon one of the observation instants of time. Thus, the narrative of the game is delayed allowing the players to continue fighting each other beyond their assumed merits of endurance, thereby letting the game to follow the path of a more realistic scenario.

In the sequel, we make assumptions on $\mathcal{A}$, $\mathcal{B}$, as being Poisson marked processes and $\mathcal{T}$ being a Poisson point process. If $X_i$ and $Y_i$ are casualties to players A and B over the interval $(\tau_{i-1},\tau_i]$, and observed upon $\tau_i$, then
\begin{align}
A_k=X_0+X_1+\dots+X_k,\,\,B_k=Y_0+Y_1+\ldots+Y_k
\end{align}
are the cumulative damages to players A and B by time $\tau_k$. With the exit indices
\begin{align}
&\nu_1:=\inf\{j\geq 0:A_j=X_0+X_1+\dots+X_j\geq
M\},\\
&\nu_2:=\inf\{k\geq 0:B_k=Y_0+Y_1+\dots+Y_k\geq
N\},\\
&\rho:=\min\left\{\nu_1,\nu_2\right\},
\end{align}
the random time $\tau_\rho$ is the observed first passage time or the observed ruin time or the observed exit time from the game. We recall that the real ruin time is unknown and it takes place anywhere between $\tau_{\rho-1}$ (observed pre-exit time) and $\tau_\rho$. Obviously, the finer are the observation times, the shorter is a delay of the end of the game. The other information of interest are $A_\rho$ and $B_\rho$ being the total damages to players A and B upon the ruin time. Clearly, $A_\rho\geq M$ or $B_\rho\geq N,$ whereas $A_{\rho-1}<M$ and $B_{\rho-1}\leq N$.

In this paper we seek the joint transforms
\begin{align}
\varPhi\left(u,v,\theta\right)=Eu^{A_\rho}v^{B_\rho}e^{-\theta\tau_\rho},\,\left\|u\right\|\leq 1,\left\|v\right\|\leq 1,\text{Re }\theta\geq 0
\end{align}
or
\begin{align}
\varPhi\left(u,v,\theta\right)=E\left[e^{-uA_\rho}v^{B_\rho}e^{-\theta\tau_\rho}\right],\,\text{Re }u\geq 0,\left\|v\right\|\leq 1,\text{Re }\theta\geq 0
\end{align}

The first transform is suited for discrete-valued components due to integer-valued marks $\{w_j\}$ and $\{z_k\}$, whereas the second transform accounts to the mixed case to be discussed next.

A method of finding $\varPhi$ was suggested by \citet{Agarwal:2005uq} in which the authors treated a multivariate marked point process with mutually dependent marks of which exactly two were so-called active. The latter means that the cumulative marks identified as active are to cross thresholds (such as $M$ and $N$ previously introduced) which bring the entire process to a hold upon crossing at the first passage time, whereas the rest of the marks identified as passive just assumes their respective values. One of them is the first passage time $\tau_\rho$. Although functional $\varPhi$ is a special case of a more general functional in \cite{Agarwal:2005uq} (that was not related to a game), we want to demonstrate the actual use of some discrete operators proposed in \cite{Agarwal:2005uq} and not only that. We also want to show that the mathematical outcome of the game is analytically tractable and numerically tame.

The following result is due to \cite{Agarwal:2005uq} in its special and a largely abridged form.

\begin{thm}[\citet{Agarwal:2005uq}] Under the assumptions (1.1)-(1.6), the functional $\varPhi$ of the process can be expressed through
\begin{align}
\gamma\left(u,v,\theta\right)=E\left[u^{X_1}v^{Y_1}e^{-\theta\tau_1}\right]
\end{align}
and it satisfies the following formula
\begin{align}
\varPhi\left(u,v,\theta\right)&=E\left[u^{A_\rho}v^{B_\rho}e^{-\theta\tau_\rho}\right]\notag
\\&=1-[1-\gamma(u,v,\theta)]\mathcal{D}_{xy}^{M-1,\,N-1}\left\{\frac{1}{1-\gamma(ux,vy,\theta)}\right\}
\end{align}

where the operator $\mathcal{D}$ (applied to a function
$\varphi:\mathbb{C}^3\rightarrow\mathbb{C}$ analytic at $(x,y)=(0,0)$) defined as

\begin{align}
\mathcal{D}^{k,m}_{x,y}\varphi\left(x,y,z\right)=\left\{
\begin{matrix}
\lim\limits_{x\rightarrow 0,y\rightarrow
0}\frac{1}{k!m!}\frac{\partial^{k+m}}{\partial x^k\partial
y^m}\left[\frac{\varphi\left(x,y,z\right)}{\left(1-x\right)\left(1-y\right)}\right],&k,m\geq 0\\
0,&k<\text{otherwise}
\end{matrix}
\right.
\end{align}
\end{thm}

\section{Motivation}

The class of antagonistic games which we study occur almost in every sphere of life. The following are some examples of games pertinent to our models.

\textbf{Cancer Treatment}. Some cancers are curable while others are not. Most metastatic cancer (which spreads from a primary site to other parts of the
body over lymph nodes and blood vessels) are incurable but can be managed to some extent using radiation alone or with other forms of treatment like chemotherapy. In relation to the antagonistic games, an oncologist, along with his/her treatment, can be regarded as player A while the tumor - as player B. The oncologist attacks the tumor cells with radiation and/or chemotherapy. While the tumor can shrink under the treatment, it may also continue spreading to other parts of the body (metastasize). Notice that any treatment by itself always has side effects (such as weakening immune system) that can be regarded as a collateral damage. At some point, when the cancer continues to spread and thus the body does not respond to the treatment, unless there are alternative options, player A is defeated. On the other hand, if the body well responds to the treatment and the tumor vanishes (the state of remission), we declare that player B is defeated. In a more modest form of a defeat, the tumor can shrink or significantly shrink instead of  disappearing entirely.

Note that cancer cells like bacteria cells typically divide in two progeny and they initially evolve as a deterministic branching process. However, some cancer cells are eliminated by T-killer cells, and at some point, when cancer matures, it evolves not from a single but many cells. If we also take into consideration mutations exhibiting an increase of the number of chromosomes (beginning in 46 to 64 and further), on an early stage, the general tumor development becomes rather chaotic allowing us to model it by an independent and stationary increment process.

Since in this paper the ``nature'' of attacks goes to integer-valued increments, and since we work on entirely discrete operational calculus, we would like to emphasize why some applications can contain entirely discrete components (or at worst they can be approximated by units made arbitrarily
small).\hfill{\qed}

\textbf{Global Military Warfare}. This is a situation where a country or group of countries are at war with one another under military operations. One classical example is the war between the United States and Japan during the WWII which consisted of multiple phases \cite{DshalalowHuang2009}. Phase 1 began with economic sanctions imposed on Japan by the US in 1940-1941 due to Japan's aggression in Manchuria. Japan tried to negotiate with the US (apparently until November 26 of 1941), but the concessions offered by the Japanese were not satisfactory to the US, and Japan not wishing to give in had no other choice as to strike on December 7, 1941. This corresponds to the beginning of phase 2. Undoubtedly, Japan was not ruined economically, but it was significantly crippled (being deprived of steel and oil, to name a few). At the same time Japan did not want to stop her campaign in China, which the US chose not to tolerate, also fearing Japan's further expansion. The Japanese Pearl Harbor attack followed by their Pacific campaign is yet another intermediate phase prior to a full scale war, because Japan believed the US will be deterred from further actions under the inflicted casualties and loss of territories in the Pacific. \hfill{\qed}

\textbf{Global Economic Warfare}. A recent economic confrontation between the US/Europe and Russia is an antagonistic game. Here player A will be the  US/Europe while player B is Russia. US and Europe stroke Russia with numerous sanctions in an attempt to weaken its economy and to drive Russia out of Ukraine, while Russia reciprocated with their own sanctions (such as forbidding US and Canada officials from entering Russia and adopting a ban on fruit, vegetables, fish, meat, and dairy products from the US and Europe) to counter such attacks. \hfill{\qed}

\textbf{Corporate Economic Hostilities}. Here we refer to a hostile relationship between two or more corporations which have a similar goal or offer similar services. In particular, we consider ride sharing companies (examples include Uber and Lyft) and taxi cabs (such as the Yellow Cab which is a sole licensed taxi cab company in Long Beach city). Uber in recent times have had to reduce their fares for their riders and this has brought about a drift of riders from Yellow Cab to Uber while they also make use of recent and flashy cars to attract its riders and make them feel more comfortable compared to Yellow Cab. This move by Uber is some form of attack on Yellow Cab which seems to be working as Uber gains more riders defecting from Yellow Cab riders. In turn, Yellow Cab attacks Uber now by calling on the authorities to make ride sharing companies face the same regulatory burdens as they do. While at the same time they are working with city councils to remove taxi's fare floor, discount fares as condition warrants, provide an ordering applications as well as getting a new branding identity. \hfill{\qed}

\textbf{Existing Literature}. The idea of utilizing multivariate random walk processes in stochastic games goes back to \citet{Agarwal:2005uq} and various earlier work of Dshalalow (see a related bibliography therein). Variants of stochastic games were studied in papers \cite{DshalalowHuang2009, DshalalowKe2009, DshalalowTreerattrakoon2008, DshalalowTreerattrakoon2010} by the first author and his collaborators, among them - games with coalitions \cite{DshalalowTreerattrakoon2008}. There were several efforts to apply formulas in \cite{Agarwal:2005uq}, such as Theorem 1 and alike, with different degree of success. One of them was Dshalalow and Treerattrakoon \cite{DshalalowTreerattrakoon2010} with continuous operational calculus. Unlike traditional methods in operational calculus and special functions, in this paper we open a new avenue of discrete operational calculus utilizing discrete inverse formulas for a class of bivariate operators $\mathcal{D}$ introduced in (1.10) which we explore in section 3. Such tools are non-existent in the literature except for a few scattered results in articles by the first author and his collaborators. We manage to obtain a fully tractable formula for the joint functional $\varPhi$ of three dependent components of the game, the first observed passage time $\tau_\rho$ and cumulative casualties of the players. We also obtain explicitly the marginal probability density function of $\tau_\rho$.

The subject of our modeling is entirely focused on fully antagonistic games which are popular in game theory. They cover a range of applications in economics \cite{Bagchi1984, BasarOlsder1982, Dockner2000, Fishburn1978, Jorgensen2004, Konstantinov2004, Shashikin2004}, warfare  \cite{Ardema1987, BasarOlsder1982, Isaacs1999, SegalMiloh1999, Shima2005}, and biology \cite{PerryRoitberg2005} to name a few. The methodology we use is based on fluctuation theory of stochastic processes as in, e.g. \cite{Agarwal:2005uq, AlzahraniDshalalow2011, Redner2001}. 

\section{A Special Case with Discrete Components}

Notice that in most applications, the functional $\gamma\left(u,v,\theta\right)=E\left[u^{X_1}v^{Y_1}e^{-\theta\tau_1}\right]$ can be readily found, as it is in our case. Let us assume that the mutual attacks on players A and B follow in accordance with two independent ordinary Poisson processes $\mathcal{A}$ and $\mathcal{B}$ specified in (1.1) of intensities $\lambda$ and $\mu$. 

Since $\mathcal{A}$ and $\mathcal{B}$ are ordinary, the respective marks $w_j$'s and $z_k$'s are 1 a.s. Furthermore, the observations take place at times $\tau_1,\tau_2,\ldots$ that forms a renewal process, with inter-renewal times $\Delta_1=\tau_1,\Delta_2=\tau_2-\tau_1,\ldots\in\left[\Delta\right],$ i.e., being identically distributed with the common Laplace-Stieltjes transform
\begin{align}
\gamma\left(\theta\right)=E\left[e^{-\theta\Delta}\right].
\end{align}

In this case, since $X$ and $Y$ are conditionally independent given $\Delta$,
\begin{align}
\gamma\left(u,v,\theta\right)=Eu^{X_1}v^{Y_1}e^{-\theta\tau_1}&=E\left[E\left[u^{X_1}v^{Y_1}e^{-\theta\tau_1}\Delta\right]\right]\notag
\\&=E\left[e^{-\theta\Delta}E\left[u^{X_1}\Delta\right]E\left[v^{Y_1}\Delta\right]\right]\notag
\\&=E\left[e^{-\theta\Delta}e^{\lambda\Delta\left(u-1\right)}e^{\mu\Delta%
\left(v-1\right)}\right]\notag
\\&=\gamma\left[\theta+\lambda-\lambda u+\mu-\mu
v\right]
\end{align}

In the special case when $\Delta\in\left[\text{Exp}\left(\gamma\right)\right]$ (exponentially distributed with parameter $\gamma$), from (3.2) we have
\begin{align}
\gamma\left(u,v,\theta\right)=\frac{\gamma}{\gamma+\lambda(1-u)+\mu(1-v)+\theta}
\end{align}
and thus
\begin{align}
\frac{1}{1-\gamma(u,v,\theta)}=1+\frac{\gamma}{\lambda(1-u)+\mu(1-v)+\theta}.
\end{align}

Now we are going to use the following properties of the $\mathcal{D}$-operator  \cite{AlzahraniDshalalow2011, DshalalowSPN2015}.

\begin{thm}[\citet{DshalalowSPN2015}]
The following properties hold true of the $\mathcal{D}$-operator introduced in (1.10).
\begin{enumerate}[(i)]
\item $\mathcal{D}_{x,y}^{k,m}=\mathcal{D}^k_x\circ              \mathcal{D}^m_y=\mathcal{D}^m_y\circ\mathcal{D}^k_x$
\item $\mathcal{D}$ is a linear functional with $\mathcal{D}_x\left\{\mathbf{1}\left(x\right)\right\}=1$, where $\mathbf{1}\left(x\right)=1$ for
all $x\in\mathbb{R}$.
\item $\mathcal{D}^k_x\left\{x^jg(x)\right\}=\mathcal{D}^{k-j}_x\{g(x)\}$.
\item For any real number $b$ it holds true that\\
$\mathcal{D}^k_x\left\{\frac{1}{1-bx}\right\}=%
\left\{
\begin{matrix}
\frac{1-b^{k+1}}{1-b},&b\neq 1\\
k+1,&b=1
\end{matrix}
\right.$
\item For any real number $a$ and for a positive integer $n$,
\begin{align*}
\mathcal{D}^k_x\left\{\frac{1}{\left(1-ax\right)^n}\right\}=\left\{
\begin{matrix}
\sum_{j=0}^k\binom{n+j-1}{j}a^j,&(a,n)\neq(1,1) \\
k+1,&(a,n)=(1,1)
\end{matrix}
\right.
\end{align*}
\item For two real numbers $a$ and $b$ it holds
\begin{align*}
\mathcal{D}^k_x\left\{\frac{1}{1-bx}\frac{1}{\left(1-ax%
\right)^n}\right\}=\left\{
\begin{matrix}
\frac{1}{1-b}\,\sum_{j=0}^k\binom{n+j-1}{j}\left(a^j-b^{k+1}\left(\frac{a}{b}%
\right)^j\right),&b\neq 1\\
\sum_{j=0}^k\binom{n+j-1}{j}a^j\left(k-j+1\right),&b=1
\end{matrix}
\right.
\end{align*}
\end{enumerate}
\end{thm}

\begin{thm}
For the special case of a discrete antagonistic game of two players, the joint functional $\varPhi$ satisfies the following formulas:
\begin{align}
\varPhi(u,v,\theta)&=\frac{\gamma}{\gamma+\lambda\left(1-u\right)+\mu\left(1-v\right)+\theta}\notag
\\&\hspace{1cm}\times\left(1-\frac{\lambda\left(1-u\right)+\mu\left(1-v\right)+\theta}{\lambda+\mu\left(1-v\right)+\theta}\psi\right)
\end{align}
\end{thm}
where
\begin{align}
\psi&=\frac{1-b^M}{1-b}-\frac{1}{1-b}\,C^N\,\,%
\sum_{j=0}^{M-1}\binom{N+j-1}{j}\left[a^j-b^M\left(\frac{a}{b}\right)^j\,%
\right]
\\a&=\frac{\lambda u}{\lambda+\mu+\theta},\hspace{1cm}b=\frac{\lambda u}{\lambda+\mu+\theta-\mu
v},\hspace{1cm}C=\frac{\mu v}{\lambda+\mu+\theta}
\end{align}
\begin{proof}
From Theorem 2$\left(i\right)$, 
\begin{align*}
\mathcal{D}_{xy}^{M-1,\,N-1}\left\{\frac{1}{1-\gamma(ux,vy,\theta)}\right\}=\mathcal{D}_x^{M-1}\left\{\mathcal{D}_y^{\,N-1}\left\{\frac{1}{1-\gamma(ux,vy,\,\theta)}\right\}\right\}.
\end{align*}
Let $p=\lambda+\mu+\theta$ then from (3.3) and Theorem 2 $\left(ii\right)$ and $\left(iv\right)$,
\begin{align*}
\mathcal{D}_y^{\,N-1}\left\{\frac{1}{1-\gamma(ux,vy,\,\theta)}\right\}&=\mathcal{D}_y^{\,N-1}\left\{1+\frac{\gamma}{p\,-\lambda ux}\,\frac{1}{1-\frac{\mu v}{p-\lambda ux}.y}\right\}
\\&=1+\frac{\gamma}{p\,-\lambda ux}\cdot\mathcal{D}_y^{\,N-1}\left\{\frac{1}{1-\frac{\mu v}{p-\lambda ux}\cdot y}\right\}
\\&=1+\frac{\gamma}{p-\lambda ux}\frac{1-\left(\frac{\mu v}{p-\lambda ux}\right)^N}{1-\frac{\mu v}{p-\lambda ux}}
\\&=1+\,\gamma\cdot\frac{1-\left[\frac{\mu
v}{p-\lambda ux}\right]^N}{p-\mu v-\lambda ux\,}
\\&=1+\frac{\gamma}{p-\mu v}\frac{1}{1-\frac{\lambda u}{p-\mu v\,}x}\left[1-\left(\frac{\mu v}{p}\frac{1}{1-\frac{\lambda u}{p\,}x}\right)^N\right]
\end{align*}
After some simple algebraic manipulation,
\begin{align*}
\mathcal{D}_y^{\,N-1}\left\{\frac{1}{1-\gamma(ux,vy,\theta)}\right\}=1+\frac{\gamma}{p-\mu v}\left[\frac{1}{1-bx}-\frac{1}{1-bx}\,C^N\frac{1}{(1-ax)^N}\right]
\end{align*}
Then
\begin{align*}
&\mathcal{D}_{xy}^{M-1,\,N-1}\left\{\frac{1}{1-\gamma(ux,vy,\theta)}\right\}
\\&=\mathcal{D}_x^{M-1}\left\{1+\frac{\gamma}{p-\mu v}\left[\frac{1}{1-bx}-C^N\frac{1}{1-bx}\,\frac{1}{(1-ax)^N}\right]\right\}
\end{align*}
From $\left(iv\right)$, and $\left(vi\right)$ with $b\neq 1$ as in case 1 of Theorem 2 $\left(vi\right)$,
\begin{align*}
&\mathcal{D}_{xy}^{M-1,\,N-1}\left\{\frac{1}{1-\gamma(ux,vy,\theta)}\right\}
\\&=1+\frac{\gamma}{p-\mu v}\left[\frac{1-b^M}{1-b}-\frac{1}{1-b}C^N\sum_{j=0}^{M-1}\binom{N+j-1}{j}\left[a^j-b^M\left(\frac{a}{b}\right)^j\right]\right]
\end{align*}
By (1.9),
\begin{align*}
\varPhi\left(u,v,\theta\right)&=1-[1-\gamma(u,v,\theta)]\mathcal{D}_{xy}^{M-1,\,N-1}\left\{\frac{1}{1-\gamma(ux,vy,\theta)}\right\}
\\&=1-\left[1-\gamma(u,v,\theta)\,\right]
\\&\hspace{1cm}\times\Bigg\{1+\frac{\gamma}{p-\mu v}\Bigg[\frac{1-b^M}{1-b}
\\&\hspace{2cm}-\frac{1}{1-b}C^N\sum_{j=0}^{M-1}\binom{N+j-1}{j}\left[a^j-b^M\left(\frac{a}{b}\right)^j\,\Bigg]\right]\Bigg\}.
\end{align*}
After some algebra and replacing $p$ with $\lambda+\mu+\theta$ and using notation (3.6)-(3.7), we find the required result,
\begin{align*}
\varPhi\left(u,v,\theta\right)&=\frac{\gamma}{\gamma+\lambda(1-u)+\mu(1-v)+\theta}
\\&\hspace{1cm}\times\left(1-\frac{\lambda(1-u)+\mu(1-v)+\theta}{\lambda+\mu(1-v)+\theta}\psi\right).
\end{align*}
\end{proof}

\begin{example}[The Marginal Transform of $\tau_\rho$]
Letting $u=v=1$ in $\varPhi\left(u,v,\theta\right)$ of Theorem 3,
\begin{align*}
E\left[e^{-\theta\tau_\rho}\right]&=\frac{\gamma}{\gamma+\theta}\left\{\left(\frac{\lambda}{\lambda+\theta}\right)^M+\left(\frac{\mu}{\lambda+\mu+\theta}\right)^N \right\}
\end{align*}
where
\begin{align*}
F=\sum_{j=0}^{M-1}\binom{N+j-1}{j}\left(\frac{\lambda}{\lambda+\mu+\theta}\right)^j\left[1-\left(\frac{\lambda}{\lambda+\theta}\right)^{M-j}\right]
\end{align*}\hfill{\qed}
\end{example}

\begin{example}[The Marginal Transform of $A_\rho$]
Letting $v=1$ and $\theta=0$ in $\varPhi\left(u,v,\theta\right)$ of Theorem 3,
\begin{align*}
E\left[u^{A_\rho}\right]=\frac{\gamma}{\gamma+\lambda\left(1-u\right)}\left[u^M+\left(\frac{\mu}{\mu+\lambda}\right)^N G\right]
\end{align*}
where
\begin{align*}
G=\sum_{j=0}^{M-1}\binom{N+j-1}{j}\left(\frac{\lambda u}{\lambda+\mu}\right)^j\left(1-u^{M-j}\,\right)
\end{align*}\hfill{\qed}
\end{example}

\begin{example}[The Marginal Transform of $B_\rho$]
Letting $u=1$ and $\theta=0$ in $\varPhi\left(u,v,\theta\right)$ of Theorem 3,
\begin{align*}
E\left[v^{B_\rho}\right]=\frac{\gamma}{\gamma+\mu\left(1-v\right)}\left[b^M+\left(\frac{\mu v}{\lambda+\mu}\right)^N H\right],
\end{align*}
where
\begin{align*}
H=\sum_{j=0}^{M-1}\binom{N+j-1}{j}\left(\frac{\lambda}{\lambda+\mu}\right)^j\left(1-b^{M-j}\right)
\end{align*}
and $b=\frac{\lambda}{\lambda+\mu\left(1-v\right)}$\hfill{\qedsymbol}
\end{example}

To find the probability density function of $\tau_\rho$ we make use of the following inverse Laplace transform formulas which can be readily proved.

\begin{lem}
Let $\alpha$, $\gamma$, $\lambda$ be some real fixed numbers. Then,
\begin{align*}
\mathcal{L}_\theta^{-1}\left(\frac{1}{\gamma+\theta}\frac{1}{\left(\lambda+\theta\right)^n}\right)\left(t\right)&=\frac{e^{-\gamma t}}{\left(\lambda-\gamma\right)^n}P\left(n-1,\left(\lambda-\gamma\right)t\right)
\end{align*}
where
\begin{align*}
P\left(n,\alpha t\right)=1-\frac{\Gamma\left(n,\alpha
t\right)}{\Gamma\left(n\right)}=1-e^{-\alpha t}\sum_{j=0}^n\frac{\left(\alpha
t\right)^j}{j!}
\end{align*}
is the regularized gamma function.
\end{lem}
\begin{proof}
\begin{align*}
\mathcal{L}_\theta^{-1}\left(\frac{1}{\gamma+\theta}\frac{1}{\left(\lambda+\theta\right)^n}\right)\left(t\right)&=e^{-\gamma t}\frac{1}{\left(\lambda-\gamma\right)^n}\left[1-e^{\left(\gamma-\lambda\right)t}\sum^{n-1}_{i=0}\frac{\left(\lambda-\gamma\right)^i}{i!}t^i\right]
\\&=\frac{e^{-\gamma t}}{\left(\lambda-\gamma\right)^n}P\left(n-1,\left(\lambda-\gamma\right)t\right).
\end{align*}
\end{proof}

\begin{lem}
Let $\alpha$, $\gamma$, $\lambda$ be some real fixed numbers. Then,
\begin{align*}
&\mathcal{L}_\theta^{-1}\left(\frac{1}{\gamma+\theta}\frac{1}{\left(\lambda+\theta\right)^m}\frac{1}{\left(\alpha+\theta\right)^n}\right)\left(t\right)
\\&=e^{-\gamma t}\left(-1\right)^n\sum^{m-1}_{k=0}\binom{n-k-1}{k}\frac{1}{\left(\lambda-\alpha\right)^{n+k}}
\\&\hspace{.25cm}\times\Bigg\{\frac{1}{\left(\lambda-\gamma\right)^{m-k}}P\left(m-k+1,\left(\lambda-\gamma\right)t\right)
\\&\hspace{.7cm}-\frac{1}{\left(\lambda-\alpha\right)^{m-k}}\sum^{n+k-1}_{i=0}\left(-1\right)^i\binom{m+i-k-1}{i}P\left(m+i-k-1,\left(\lambda-\alpha\right)t\right)\Bigg\}.
\end{align*}
\end{lem}
\begin{proof}
By well-known formulas for inverse Laplace transforms,
\begin{align*}
&\mathcal{L}_\theta^{-1}\left(\frac{1}{\gamma+\theta}\frac{1}{\left(\lambda+\theta\right)^m}\frac{1}{\left(\alpha+\theta\right)^n}\right)\left(t\right)
\\&=g\left(\alpha,\lambda\right)\Bigg\{\frac{1}{\left(\lambda-\gamma\right)^{m-k}}\Bigg[e^{-\gamma t}-e^{-\lambda t}\sum^{m-k-1}_{r=0}\frac{\left(\lambda-\gamma\right)^r}{r!}t^r\Bigg]
\\&\hspace{2cm}-\sum^{n+k-1}_{i=0}\left(\alpha-\lambda\right)^i\Bigg[\binom{m+i-k-1}{i}\frac{1}{\left(\lambda-\alpha\right)^{m+i-k}}
\\&\hspace{3cm}\times\left(e^{-\gamma t}-e^{\left(\alpha-\lambda-\gamma\right)t}\sum^{m+i-k-1}_{s=0}\frac{\left(\lambda-\alpha\right)^s}{s!}t^s\right)\Bigg]\Bigg\}
\end{align*}
where 
\begin{align*}
g\left(\alpha,\lambda\right)&=\frac{1}{\left(n-1\right)!\left(\alpha-\lambda\right)^n}\sum^{m-1}_{k=0}\frac{\left(n+k-1\right)!}{k!}\frac{1}{\left(\lambda-\alpha\right)^k}
\\&=\left(-1\right)^n\sum^{m-1}_{k=0}\binom{n-k-1}{k}\frac{1}{\left(\lambda-\alpha\right)^{n+k}}.
\end{align*}
Next, we see this is equivalent to
\begin{align*}
&e^{-\gamma t}g\left(\alpha,\lambda\right)\Bigg\{\frac{1}{\left(\lambda-\gamma\right)^{m-k}}\Bigg[1-e^{-\left(\lambda-\gamma\right)t}\sum^{m-k-1}_{r=0}\frac{\left(\lambda-\gamma\right)^r}{r!}t^r\Bigg]
\\&\hspace{2cm}-\frac{1}{\left(\lambda-\alpha\right)^{m-k}}\sum^{n+k-1}_{i=0}\left(-1\right)^i\binom{m+i-k-1}{i}
\\&\hspace{5cm}\times\left(1-e^{-\left(\lambda-\alpha\right)t}\sum^{m+i-k-1}_{s=0}\frac{\left(\lambda-\alpha\right)^s}{s!}t^s\right)\Bigg\}
\\&=e^{-\gamma t}\left(-1\right)^n\sum^{m-1}_{k=0}\binom{n-k-1}{k}\frac{1}{\left(\lambda-\alpha\right)^{n+k}}
\\&\hspace{2cm}\times\biggl\{\frac{1}{\left(\lambda-\gamma\right)^{m-k}}P\left(m-k+1,\left(\lambda-\gamma\right)t\right)
\\&\hspace{3cm}-\frac{1}{\left(\lambda-\alpha\right)^{m-k}}\sum^{n+k-1}_{i=0}\left(-1\right)^i\binom{m+i-k-1}{i}
\\&\hspace{6cm}\times P\left(m+i-k-1,\left(\lambda-\alpha\right)t\right)\biggr\}
\end{align*}
\end{proof}

\begin{example}[The Probability Density Function of $\tau_\rho$]
Revisiting Example 1 we go further to obtain the probability density function of the first observed passage time of the game end. Since $Ee^{-\theta\tau_\rho}$ is the Laplace-Stieltjes transform, we need to divide it by $\theta$ and then take the Laplace inverse to obtain the density function, $f_{\tau_\rho}$,
\begin{align*}
f_{\tau_\rho}(t)&=\mathcal{L}_\theta^{-1}\left\{E\left[e^{-\theta\tau_\rho}\right]\right\}(t)
\\&=\mathcal{L}_\theta^{-1}\Biggl\{\frac{\gamma\lambda^M}{(\gamma+\,\theta)(\lambda+\,\theta)^M}\,+\,\,\sum_{j=0}^{M-1}\binom{N+j-1}{j}\frac{\gamma\lambda^j\mu^N}{(\gamma+\,\theta)(\lambda+\mu+\theta)^{N+j}}
\\&\hspace{1.5cm}-\sum_{j=0}^{M-1}\binom{N+j-1}{j}\frac{\gamma\lambda^M\mu^N}{(\gamma+\,\theta)(\lambda+\,\theta)^{M-j}(\lambda+\mu+\,\theta)^{N+j}}\Biggr\}\left(t\right).
\end{align*}
By Lemma 4 and Lemma 5 and after some algebra, we have
\begin{align*}
f_{\tau_\rho}(t)&=\frac{\gamma\lambda^Me^{-\gamma t}}{\left(\lambda-\gamma\right)^M}P\left(M-1,\lambda-\gamma\right)
\\&+\gamma\mu^Ne^{-\gamma t}\sum_{j=0}^{M-1}\binom{N+j-1}{j}\frac{\lambda^j}{\left(\lambda+\mu-\gamma\right)^{N+j}}P\left(N+j-1,\lambda+\mu-\gamma\right)
\\&-\gamma\mu^Ne^{-\gamma t}\sum_{j=0}^{M-1}\binom{N+j-1}{j}\lambda^{M-N-j}\sum_{k=0}^{M-j-1}\binom{N+j-k-1}{k}\frac{\left(-1\right)^k}{\lambda^k}
\\&\hspace{.5cm}\times\biggl\{\frac{1}{\left(\lambda-\gamma\right)^{M-j-k}}P\left(M-j-k+1,\left(\lambda-\gamma\right)t\right)-\frac{\left(-1\right)^{M-j-k}}{\mu^{M-j-k}}
\\&\hspace{1cm}\times\hspace{-.25cm}\sum_{l=0}^{N+j+k-1}\hspace{-.25cm}\left(-1\right)^l\binom{M-j+l-k-1}{l}P\left(M-j+l-k-1,-\mu t\right)\biggr\}.
\end{align*}
\end{example}

\bibliographystyle{plainnat}
\bibliography{DOCDSG}
\end{document}